\documentclass[11pt,a4paper]{article}
\usepackage[left=2cm,right=2cm,top=2cm,bottom=2.cm]{geometry}
\usepackage{a4wide}
\usepackage[utf8]{inputenc}
\usepackage[english]{babel}
\usepackage{amsmath}
\usepackage{amsfonts}
\usepackage{amssymb}
\usepackage{amsthm}
\usepackage{euscript}

\usepackage{xcolor}
\usepackage{color,url}
\usepackage{float}
\usepackage{hyperref}
\hypersetup{
	colorlinks=true,
	linkcolor=blue, 
	linktoc=all
}

\usepackage{tikz}
\usetikzlibrary{arrows}
\definecolor{wine}{rgb}{2,2,251}
\definecolor{myblue}{rgb}{0.2,0.4,2}

\usepackage{graphicx, graphics}
\usepackage{graphics,graphicx}
\usepackage{epsf}
\usepackage{cite}
\usepackage{pdflscape}
\usepackage{array}
\usepackage{float}
\usepackage{setspace}
\newcolumntype{P}[1]{>{\arraybackslash}p{#1}}
\newcolumntype{M}[1]{>{\centering\arraybackslash}m{#1}}
\newcolumntype{C}[1]{%
	>{\vbox to 1ex\bgroup\vfill\centering}%
	p{#1}%
	<{\egroup}}  
\usepackage[affil-it]{authblk}
\newcolumntype{L}[1]{>{\raggedright\arraybackslash}p{#1}}

\theoremstyle{plain}
\newtheorem{theorem}{Theorem}
\newtheorem{corollary}{Corollary}

\newtheorem{proposition}{Proposition}

\theoremstyle{definition}
\newtheorem{definition}{Definition}
\newtheorem{example}{Example}

\usepackage{multirow}
\allowdisplaybreaks
\usepackage{booktabs}

\newcommand{\rec}{\fboxsep0pt\fbox{\rule{1.2em}{0pt}\rule{0pt}{1.2ex}}}
\newcommand{\squ}{\fboxsep0pt\fbox{\rule{0.6em}{0.1pt}\rule{0.1pt}{1.2ex}}}
\newcommand{\T}{\mathcal{T_{FF}}}

\newcommand{\Tt}[2]{\mathcal{F}(#1,#2)}

\DeclareRobustCommand{\stirling}{\genfrac\{\}{0pt}{}}
\newcommand{\genstirlingI}[3]{	\genfrac{[}{]}{0pt}{#1}{#2}{#3} }
\newcommand{\genstirlingII}[3]{	\genfrac{\{}{\}}{0pt}{#1}{#2}{#3} }
\newcommand{\stirlingI}[2]{\genstirlingI{}{#1}{#2}}

\newcommand{\stirlingII}[2]{\genstirlingII{}{#1}{#2}}

\title{\textbf{The Fibonacci-Fubini and Lucas-Fubini numbers}}

\author{Yahia Djemmada$^1$,  
	Abdelghani Mehdaoui$^2$, 	
	L\'aszl\'o N\'emeth$^3$, and 
	L\'aszl\'o Szalay$^4$ \thanks{
		Y.~D.:  ORCID 0000-0002-9425-0216, 
		M.~A.: ORCID 0000-0002-8080-5719, 
		L.~N.: ORCID 0000-0001-9062-9280, 
		L.~Sz.: ORCID 0000-0002-4582-6100.}   
	}

\affil{$^{1,2}$National Higher School of Mathematics, P.O.Box 75, Mahelma 16093, Sidi Abdellah, Algiers, Algeria. \\$^{1,2}$USTHB,  Faculty of Mathematics, RECITS Laboratory BP 32, El Alia 16111 \\ Bab Ezzouar, Algiers, Algeria. \\
	{yahia.djemmada@nhsm.edu.dz} and {abdelghani.mehdaoui@nhsm.edu.dz} }

\affil{$^{3,4}$Institute of Basic Sciences, 
	University of Sopron \\Bajcsy-Zs. u. 4, H-9400, Sopron, Hungary \\ 
	{nemeth.laszlo@uni-sopron.hu} and {szalay.laszlo@uni-sopron.hu} }

\affil{$^4$Department of Mathematics, J. Selye University, Hradn\'a str.\  21, \\ 945 01 Kom\'arno, Slovakia }

\usepackage{datetime}

\date{\today}
\begin{document}
	\definecolor{zzttqq}{rgb}{0.6,0.2,0.}
	\definecolor{xdxdff}{rgb}{0.49019607843137253,0.49019607843137253,1.}
	\definecolor{uuuuuu}{rgb}{0.26666666666666666,0.26666666666666666,0.26666666666666666}
	\maketitle

		
		\begin{abstract}
			Based on the combinatorial interpretation of the ordered Bell numbers, which count all the ordered partitions of the set $[n]=\{1,2,\dots,n\}$, we introduce the Fibonacci partition as a Fibonacci permutation of its blocks. Then we define the Fibonacci-Fubini numbers that count the total number of Fibonacci partitions of $[n]$.
			
			We study the classical properties of this sequence (generating function, explicit and Dobi\'nski-like formula,  etc.), we give combinatorial interpretation, and we extensively examine the Fibonacci-Fubini arithmetic triangle. We give some associate linear recurrence sequences, where in some sequences the Stirling numbers of the first and second kinds appear as well.
			\\[1mm]
			{	
				{\em Key Words: Ordered partitions, Fubini numbers, Fibonacci permutations, Lucas permutations, Stirling numbers. }\\
				{\em MSC 2020: 	05A15, 05A18, 11B39, 11B37, 11B73. }\\ 
			}
		\end{abstract}
		
		
		\section{Introduction}
		
		In enumerative combinatorics, a \textit{partition} of a set $[n] := \{1, 2, \dots ,n\}$ is a distribution of its elements into $k$ non-empty disjoint subsets $B_1|B_2|\dots|B_k$, called \textit{blocks}. (We separate the subsets by bars $"|"$.) We assume that the blocks are arranged in ascending order according to their minimum elements ($\min B_1<\min B_2<\cdots<\min B_k$).
		
		The Stirling numbers of the second kind, denoted by ${n \brace k}$, count the number of partitions of $[n]$ into exactly $k$ blocks \cite{Com}. Table~\ref{tab:Stirling} shows the first few items of the Stirling numbers of the second kind in triangular form.
		\begin{table}[H]
			\centering
			\begin{tabular}{c|cccccccc}
				\toprule
				${n \backslash k}$ & 0 & 1 & 2  & 3   & 4   & 5   & 6  & 7 \\
				\midrule
				0                  & 1 &  &   &   &    &    &   &  \\
				1                  & 0 & 1 &   &  & & && \\
				2                  & 0 & 1 & 1  &    &    &    &   &  \\
				3                  & 0 & 1 & 3  & 1   &    &    &   &  \\
				4                  & 0 & 1 & 7  & 6   & 1   &    &   &  \\
				5                  & 0 & 1 & 15 & 25  & 10  & 1   &   &  \\
				6                  & 0 & 1 & 31 & 90  & 65  & 15  & 1  & \\
				7                  & 0 & 1 & 63 & 301 & 350 & 140 & 21 & 1 \\
				
			\end{tabular}
			\caption{Stirling numbers of the second kind ${n \brace k}$.}
			\label{tab:Stirling}
		\end{table}
		The Stirling numbers of the second kind satisfy the recurrence relation
		\begin{equation}\label{StiR}
			{n\brace k}={n-1\brace k-1}+k{n-1\brace k}, \quad \text{for all } 1\leq k \leq n,
		\end{equation}
		with initial condition ${n\brace n}=1$ for $0\leq n$ and ${n\brace 0}={0\brace n}=0$ for $0<n $.
		Their explicit formula is given by
		\begin{align}\label{StiEx}
			{n \brace k}=\frac{1}{k!} \sum_{j=0}^{k}(-1)^{k-j}\binom{k}{j}j^n,
		\end{align}
		and their exponential generating function is
		\begin{equation}\label{eq:gen_Strirling}
			\sum_{n=k}^\infty{n \brace k}\frac{t^n}{n!}= \frac{(e^{t}-1)^k}{k!}.
		\end{equation}

		Expanding on this, authors often explore the Fubini numbers or ordered Bell numbers (A000670 in the OEIS\cite{oeis}), denoted as $w_n$, which signify the total number of partitions of $n$ distinguishable elements into different sizes of distinguishable blocks.
		Mathematically, this is expressed as
		\begin{equation*}
			w_n=\sum_{k=0}^{n}k!{n \brace k}, \quad \text{for all } 0 \leq n.
		\end{equation*}
		This formula is justified combinatorially by considering $k!$ distinguishable blocks due to their ordered nature, and ${n \brace k}$ denotes the number of partitions into exactly $k$ blocks. The summation over $k$ captures the total number of partitions of varying sizes.
		
		Regarding the notation, throughout this paper we represent the elements of the same block by adjacent numbers, and we separate the blocks by bars $"|"$.
		\begin{example}
			The partitions of the set $[3]=\{1,2,3\}$ are
			\[{123};\quad{1|23};\quad{12|3};\quad{13|2};\quad{1|2|3}.\]
			Thus $\stirlingII{3}{1}=1$, $\stirlingII{3}{2}=3$, $\stirlingII{3}{3}=1$, and its ordered partitions are the permutations of all the partitions above, $w_3=13$, as follows.
			\begin{equation*}
				\begin{aligned}
					{123};&&{1|23};&&{23|1};&&{12|3};&&{3|12};&&{13|2};&&{2|13};&&{1|2|3};&&{1|3|2};&&{2|1|3};&&{2|3|1};&&{3|1|2};&&{3|2|1}. 
				\end{aligned}
			\end{equation*}
		\end{example}

		Belbachir, Djemmada, and N\'emeth \cite{HacGjeNem} introduced a new type of ordered Bell numbers using the derangement numbers. Partly inspired by this, in this article, we define two new types of sequences, the Fibonacci-Fubini and Lucas-Fubini sequences, based on the Fubini sequence. We give their most relevant properties: the exponential generator functions, the explicit, Dobi\'nski-like, and derivative forms. Moreover, we give their combinatorial interpretations, and we extensively examine the Fibonacci-Fubini arithmetic triangle. We give some associate constant coefficient linear recurrence sequences, where in some sequences the Stirling numbers of the first kind appear as coefficients.

		\section{The Fibonacci-Fubini numbers}
		
		Let $S_n$ denote the set of all $n$-permutations of the set $[n]$.
		A permutation $\sigma \in S_n$, represented as $\sigma=\sigma_1\sigma_2\cdots\sigma_n$, is termed a {\it Fibonacci permutation} if it satisfies the condition $|\sigma_i-i|\leq 1$ for each element $i$, see \cite{Foundas,Fou}.
		In other words, each element $i$ in a Fibonacci permutation should either be in its original position or in a position adjacent to $i$ (either $i-1$ or $i+1$). For instance, the $4$-permutation $1243$ is a Fibonacci permutation, whereas $1423$ is not because $(\sigma_2-2>1)$.
		
		Foundas et al. \cite{Fou,Foundas} demonstrated that the count of Fibonacci permutations aligns with the well-known Fibonacci sequence. The authors established that the number of $n$-permutations of Fibonacci type is equal to $F_{n+1}$, where $F_n$ defined by recurrence $F_n=F_{n-1}+F_{n-2}$ for $2\leq n$, and with initial values $F_0=0$, $F_1=1$. Moreover, it is sequence {A000045} in OEIS \cite{oeis}.

		Additionally, the number of Fibonacci permutations can be expressed using Binet's formula given by
		\begin{equation}\label{FibB}
			F_n=\frac{\alpha^{n}-\beta^{n}}{\alpha-\beta},
		\end{equation}
		where $\alpha=(1+\sqrt{5})/2$, $\beta=(1-\sqrt{5})/2$, and $\alpha-\beta=\sqrt{5}$. The Binet's formula of the well-known Lucas sequence $L_n$ (defined by $L_n=L_{n-1}+L_{n-2}$ for $2\leq n$, $L_0=2$, $L_1=1$, and {A000032} in OEIS \cite{oeis}) is
		\begin{equation}\label{LucB}
			L_n=\alpha^{n}+\beta^{n}.
		\end{equation}
		
		In our current research, we focus on exploring the number of Fibonacci-ordered partitions. In other words, we examine the arrangements of blocks such that each block is either in its original position or in one of the positions of its neighbors.
		\begin{definition}
			A Fibonacci ordered partition $\pi_F$ of the set $[n]$ is a Fibonacci permutation of its blocks $B_1|B_2|\ldots|B_k$. We represent this as \[\pi_F([n])=B_{\sigma_1}|B_{\sigma_2}|\ldots|B_{\sigma_k},\]
			where $\sigma=\sigma_1\sigma_2\cdots\sigma_k$ is a Fibonacci permutation.
		\end{definition}
		Let Fibonacci-Fubini numbers be the total numbers of Fibonacci-ordered partitions, denoted $\mathcal{F}_n$.
		\begin{proposition}
			The Fibonacci-Fubini numbers $(\mathcal{F}_n)_{n\geq 0}$ are given by
			\begin{equation}\label{eq:Fibo-Stir}
				\mathcal{F}_n=\sum_{k=0}^n F_{k+1}{n \brace k}.
			\end{equation}
		\end{proposition}
		\begin{proof}
			Using a combinatorial argument, the number of partitions having exactly $k$ blocks is ${n \brace k}$; the number of Fibonacci permutations of the blocks is $F_{k+1}$. Therefore, the number of Fibonacci ordered partitions is $F_{k+1}{n\brace k}$. Summing it over $k$ we count the total number of Fibonacci ordered partitions, which yields the result.
		\end{proof}
		
		Based on equality \eqref{eq:Fibo-Stir}, we define the Lucas-Fubini numbers $(\mathcal{L}_n)_{n\geq 0}$  by
		\begin{equation}\label{eq:Lucas-Stir}
			\mathcal{L}_n=\sum_{k=0}^n L_{k+1}{n \brace k}.
		\end{equation}
		
		Computing the first few item of the Fubini, Fibonacci-Fubini and Lucas-Fubini sequences we obtain 
		\begin{align*}
			(w_n)_{n\geq 0} =& (1, 1, 3, 13, 75, 541, 4683, 47293, 545835, 7087261, 102247563,  \ldots),\\
			(\mathcal{F}_n)_{n\geq 0} =& (1,1,3,10,38,164,791,4194,24138, 149423, 988033, 6939682, \ldots),\\
			(\mathcal{L}_n)_{n\geq 0} =& (1,3,7,22,84,366,1771,9388,53990,334093,2209039,15516716, \ldots).
		\end{align*}
		
		In the last section, we give a detailed combinatorial interpretation of the Fibonacci-Fubini number.

		\section{Some properties of Fibonacci-Fubini and Lucas-Fubini numbers}
		
		In this section, we give some fundamental properties of Fibonacci-Fubini and Lucas-Fubini numbers.
		
		\subsection{Explicit formulas}
		
		\begin{theorem}
			The explicit formula for the Fibonacci-Fubini numbers $\mathcal{F}_n$ is given by
			\begin{align*}
				\mathcal{F}_n = \frac{1}{\sqrt{5}} \sum_{k=0}^{n}\sum_{j=0}^{k}(-1)^{k-j}\binom{k}{j} \frac{\alpha^{k+1}-\beta^{k+1}}{k!}  j^n.
			\end{align*}
		\end{theorem}
		\begin{proof} Using the explicit formula of Stirling numbers \eqref{StiEx} and the Binet formula \eqref{FibB}, we have 
			\begin{align*}
				\mathcal{F}_n=	\sum_{k=0}^{n}F_{k+1}\; {n \brace k}=&\sum_{k=0}^{n}\frac{\alpha^{k+1}-\beta^{k+1}}{\alpha - \beta }\cdot \frac{1}{k!}  \sum_{j=0}^{k}(-1)^{k-j}\binom{k}{j}j^n.
			\end{align*}
			Since $\alpha - \beta=\sqrt{5}$, we obtain the statement.
		\end{proof}
		Using the explicit formula of Stirling numbers \eqref{StiEx} and the Binet formula \eqref{LucB} one can prove similarly the following theorem.
		\begin{theorem}
			The explicit formula of the Lucas-Fubini numbers $\mathcal{L}_n$ is given by 
			\begin{align*}
				\mathcal{L}_n = \sum_{k=0}^{n}\sum_{j=0}^{k}(-1)^{k-j}\binom{k}{j}\frac{\alpha^{k+1}+\beta^{k+1}}{k!}  j^n.
			\end{align*}
		\end{theorem}
		
		\subsection{Exponential generating functions}
		
		Let $\EuScript{F}(t)$ and $\EuScript{L}(t)$ be the exponential generator functions of Fibonacci-Fubini and Lucas-Fubini numbers, respectively. Then 
		\[\EuScript{F}(t)= \sum_{n= 0}^\infty \mathcal{F}_n \frac{t^n}{n!} \quad \text{ and }\quad \EuScript{L}(t)= \sum_{n= 0}^\infty \mathcal{L}_n \frac{t^n}{n!}.\] 
		\begin{theorem}\label{thm:FibSt}
			The exponential generating function of the sequence $\mathcal{F}_n$ is given by
			\begin{align*}
				\EuScript{F}(t)= \frac{1}{\sqrt{5}}\left(\alpha e^{\left(e^{t}-1\right)\alpha}-\beta e^{\left(e^{t}-1\right)\beta}\right).
			\end{align*}
		\end{theorem}
		\begin{proof} Using the Binet formula \eqref{FibB} for the Fibonacci numbers, $\stirlingII{n}{k}=0$ for $n<k$,  and the generating function \eqref{eq:gen_Strirling}
			we have
			\begin{align*}
				\EuScript{F}(t) 
				=&\sum_{n\geq 0}\sum_{k=0}^{n}F_{k+1}\; {n \brace k}\frac{t^n}{n!} =\sum_{n\geq 0}\sum_{k=0}^{n}\frac{\alpha^{k+1}-\beta^{k+1}}{\alpha-\beta}{n \brace k} \frac{t^n}{n!}\\
				=&\frac{1}{\alpha-\beta}\sum_{k\geq 0}\left(\alpha^{k+1}\sum_{n\geq 0}{n \brace k}\frac{t^n}{n!} - \beta^{k+1}\sum_{n\geq 0}{n \brace k}\frac{t^n}{n!}\right)\\
				=&\frac{1}{\alpha-\beta}\sum_{k\geq 0}\left(\alpha^{k+1} \frac{(e^t-1)^k}{k!} - \beta^{k+1}\frac{(e^{t}-1)^k}{k!}\right)\\
				=&\frac{1}{\alpha-\beta}\left(\alpha \sum_{k\geq 0}\frac{\left(\alpha (e^t-1)\right)^k}{k!} - \beta \sum_{k\geq 0}\frac{\left(\beta (e^t-1)\right)^k}{k!}\right)\\
				=&\frac{1}{\sqrt{5}}\left(\alpha e^{\left(e^{t}-1\right)\alpha}-\beta e^{\left(e^{t}-1\right)\beta}\right).
			\end{align*}
		\end{proof}

		\begin{theorem}\label{thm:LucSt}
			The exponential generating function for the $\mathcal{L}_n$  numbers is given by
			\begin{align}
				\EuScript{L}(t)=  \alpha e^{\left(e^{t}-1\right)\alpha}+\beta e^{\left(e^{t}-1\right)\beta}.
			\end{align}
		\end{theorem}
		\begin{proof} Using the Binet formula \eqref{LucB} for the Lucas numbers, we have
			\begin{align*}
				\EuScript{L}(t)
				=&	\sum_{n\geq 0}\sum_{k=0}^{n}L_{k+1}\; {n \brace k}\frac{t^n}{n!} =\sum_{n\geq 0}\sum_{k=0}^{n}\left(\alpha^{k+1}+\beta^{k+1}\right){n \brace k} \frac{t^n}{n!}\\
				=&\sum_{k\geq 0}\left(\alpha^{k+1}\sum_{n\geq 0}{n \brace k}\frac{t^n}{n!} + \beta^{k+1}\sum_{n\geq 0}{n \brace k}\frac{t^n}{n!}\right)\\
				=&\sum_{k\geq 0}\left(\alpha^{k+1} \frac{(e^t-1)^k}{k!} + \beta^{k+1}\frac{(e^{t}-1)^k}{k!}\right)\\
				=&\alpha \sum_{k\geq 0}\frac{\left(\alpha (e^t-1)\right)^k}{k!} + \beta \sum_{k\geq 0}\frac{\left(\beta (e^t-1)\right)^k}{k!}\\
				=&\alpha e^{\left(e^{t}-1\right)\alpha}+\beta e^{\left(e^{t}-1\right)\beta}.
			\end{align*}
		\end{proof}

		\subsection{Dobi\'nski-like formulas}
		In this subsection, we give the Dobi\'nski-like formulas of the Fibonacci-Fubini and Lucas-Fubini numbers. Dobi\'nski \cite{Dob} gave  the useful formula $(1/e)\sum_{k=0}^{\infty} (k^n/k!)$ for Bell numbers.

		\begin{theorem} For all $n\geq 0$ the Dobi\'nski-like formula of Fibonacci-Fubini numbers is  
			\begin{align*}
				\mathcal{F}_n=\frac{1}{\sqrt{5}}\left(\frac{\alpha}{e^{\alpha}}\sum_{k\geq 0} \frac{\alpha^k k^n}{k!}-\frac{\beta}{e^{\beta}}\sum_{k\geq 0} \frac{\beta^k k^n}{k!}\right).
			\end{align*}
		\end{theorem}
		\begin{proof}
			From Theorem \ref{thm:FibSt}, we have		
			\begin{align*}
				\sum_{n\geq 0}\mathcal{F}_n\frac{t^n}{n!}=&\frac{1}{\sqrt{5}}\left(\alpha e^{\left(e^{t}-1\right)\alpha}-\beta e^{\left(e^{t}-1\right)\beta}\right)\\
				=& \frac{1}{\sqrt{5}}\left(\alpha e^{\alpha e^{t}}e^{-\alpha}-\beta e^{\beta e^{t}}e^{-\beta} \right)\\
				=& \frac{1}{\sqrt{5}}\left(\frac{\alpha}{e^{\alpha}}\sum_{k\geq 0}\frac{\alpha^k e^{kt}}{k!}-\frac{\beta}{e^{\beta}}\sum_{k\geq 0}\frac{\beta^k e^{kt}}{k!}\right)\\
				=&\frac{1}{\sqrt{5}}\left(\frac{\alpha}{e^{\alpha}}\sum_{k\geq 0}\frac{\alpha^k}{k!}\sum_{n\geq 0}\frac{k^nt^n}{n!}-\frac{\beta}{e^{\beta}}\sum_{k\geq 0}\frac{\beta^k}{k!}\sum_{n\geq 0}\frac{k^nt^n}{n!}\right)\\
				=& \frac{1}{\sqrt{5}}\left(\frac{\alpha}{e^{\alpha}}\sum_{n\geq 0}\sum_{k\geq 0}\frac{\alpha^k k^n}{k!}\frac{t^n}{n!}-\frac{\beta}{e^{\beta}}\sum_{n\geq 0}\sum_{k\geq 0}\frac{\beta^k k^n}{k!}\frac{t^n}{n!}\right)\\
				=&\sum_{n\geq 0} \frac{1}{\sqrt{5}}\left(\frac{\alpha}{e^{\alpha}}\sum_{k\geq 0}\frac{\alpha^k k^n}{k!}-\frac{\beta}{e^{\beta}}\sum_{k\geq 0}\frac{\beta^k k^n}{k!}\right)\frac{t^n}{n!}.
			\end{align*}
			By comparing the coefficients of ${t^n}/{n!}$ on both sides we get the result.
		\end{proof}
		\begin{theorem}
			The Dobinski-like formula for the Lucas-Fubini numbers is 
			\begin{align}
				\mathcal{L}_n=\frac{\alpha}{e^{\alpha}}\sum_{k\geq 0} \frac{\alpha^k k^n}{k!}+\frac{\beta}{e^{\beta}}\sum_{k\geq 0} \frac{\beta^k k^n}{k!}.
			\end{align}
		\end{theorem}
		\begin{proof}
			Similarly, from Theorem \ref{thm:LucSt}, we have		
			\begin{align*}
				\sum_{n\geq 0}\mathcal{L}_n\frac{t^n}{n!} =&  \alpha e^{\left(e^{t}-1\right)\alpha}+\beta e^{\left(e^{t}-1\right)\beta}\\
				=& \alpha e^{\alpha e^{t}}e^{-\alpha}+\beta e^{\beta e^{t}}e^{-\beta} \\
				=& \frac{\alpha}{e^{\alpha}}\sum_{k\geq 0}\frac{\alpha^k e^{kt}}{k!}+\frac{\beta}{e^{\beta}}\sum_{k\geq 0}\frac{\beta^k e^{kt}}{k!}\\
				=&\frac{\alpha}{e^{\alpha}}\sum_{k\geq 0}\frac{\alpha^k}{k!}\sum_{n\geq 0}\frac{k^nt^n}{n!}+\frac{\beta}{e^{\beta}}\sum_{k\geq 0}\frac{\beta^k}{k!}\sum_{n\geq 0}\frac{k^nt^n}{n!}\\
				=& \frac{\alpha}{e^{\alpha}}\sum_{n\geq 0}\sum_{k\geq 0}\frac{\alpha^k k^n}{k!}\frac{t^n}{n!}+\frac{\beta}{e^{\beta}}\sum_{n\geq 0}\sum_{k\geq 0}\frac{\beta^k k^n}{k!}\frac{t^n}{n!}\\
				=&\sum_{n\geq 0}\left( \frac{\alpha}{e^{\alpha}}\sum_{k\geq 0}\frac{\alpha^k k^n}{k!}+\frac{\beta}{e^{\beta}}\sum_{k\geq 0}\frac{\beta^k k^n}{k!}\right)\frac{t^n}{n!}.
			\end{align*}
			The comparison of the coefficients of ${t^n}/{n!}$ on the two sides implies the result.
		\end{proof}

		\subsection{Higher order derivatives}
		
		In this subsection, derivative of order $r$ of exponential generating function of the Fibonacci-Fubini and Lucas-Fubini numbers.

		\begin{theorem} \label{th:deivate_Fibo} The $r$th higher derivative of the exponential generating function of the Fibonacci-Fubini numbers is given by 
			\begin{align*}
				\EuScript{F}^{(r)}(t)= \frac{1}{\sqrt{5}} \sum_{i=1}^{r}e^{it}{r \brace i} \left(\alpha^{i+1} e^{\alpha (e^t-1)}-\beta^{i+1} e^{\beta (e^t-1)}\right).
			\end{align*}
		\end{theorem}
		\begin{proof}We prove the theorem using induction on $r$.
			For $r=1$, \[
			\EuScript{F}^{(1)}(t)= \frac{1}{\sqrt{5}} e^{t}\left(\alpha^{2} e^{\left(e^{t}-1\right) \alpha} - \beta^{2} e^{\left(e^{t}-1\right) \beta}\right)=\EuScript{F}'(t).
			\]
			We suppose the equation is true for $r-1$ and show it holds for $r$.
			\[
			\begin{aligned}
				\left(\EuScript{F}^{(r-1)}(t)\right)^{\prime} & =\left( \frac{1}{\sqrt{5}} \sum_{i=1}^{r-1} e^{i t}{r-1 \brace i}\left(\alpha^{i+1} e^{\left(e^{t}-1\right) \alpha}- \beta^{i+1} e^{\left(e^{t}-1\right) \beta}\right)\right)^{\prime} \\
				\EuScript{F}^{(r)}(t)           & = \frac{1}{\sqrt{5}}\sum_{i=1}^{r-1}{r-1 \brace i}\left(\alpha^ { i + 1 } e^{e^{(t-1) \alpha}+i t}- \beta^{i+1}e^{e^{(t-1) \beta}+i t}\right)^{\prime}       \\
				& = \frac{1}{\sqrt{5}} \sum_{i=1}^{r-1}{r-1 \brace i}\left(\left(\alpha e^{t}+i\right) \alpha^{i+1} e^{\left(e^{t}-1\right) \alpha+i t}- \left(\beta e^{t}+i\right) \beta^{i+1} e^{\left(e^{t}-1\right) \beta+i t}\right)\\
				&=\frac{1}{\sqrt{5}} \left(\sum_{i=1}^{r-1}{r-1 \brace i}\left(e^{(i+1)t} \alpha^{i+2} e^{\left(e^{t}-1\right) \alpha}- e^{(i+1)t} \beta^{i+2} e^{\left(e^{t}-1\right) \beta}\right)\right.\\
				&\hspace*{1cm}+\left.\sum_{i=1}^{r-1}i{r-1 \brace i}e^{it} \left(\alpha^{i+1} e^{\left(e^{t}-1\right) \alpha}- \beta^{i+1} e^{\left(e^{t}-1\right) \beta}\right)\right)\\
			\end{aligned}\]	\[ \begin{aligned}
				&= \frac{1}{\sqrt{5}}\left(\sum_{i=2}^{r}{r-1 \brace i-1}\left(e^{it} \alpha^{i+1} e^{\left(e^{t}-1\right) \alpha}- e^{it} \beta^{i+1} e^{\left(e^{t}-1\right) \beta}\right)\right.\\&
				\hspace*{1cm}+\left.\sum_{i=1}^{r-1}i{r-1 \brace i}e^{it} \left(\alpha^{i+1} e^{\left(e^{t}-1\right) \alpha}- \beta^{i+1} e^{\left(e^{t}-1\right) \beta}\right)\right) \\
				&=\frac{1}{\sqrt{5}}\sum_{i=1}^{r}\left({r-1 \brace i-1}+i{r-1 \brace i}\right)\left(e^{it} \left(\alpha^{i+1} e^{\left(e^{t}-1\right) \alpha}- \beta^{i+1} e^{\left(e^{t}-1\right) \beta}\right)\right)\\
				&=\frac{1}{\sqrt{5}} \sum_{i=1}^{r}{r \brace i}e^{it} \left(\alpha^{i+1} e^{\left(e^{t}-1\right) \alpha}- \beta^{i+1} e^{\left(e^{t}-1\right) \beta}\right).
			\end{aligned}
			\]
			Recall that $\stirlingII{r-1}{i-1}=0$ when $i=1$.
		\end{proof}
		
		\begin{theorem} The $r$th higher derivative of the exponential generating function of the Lucas-Fubini numbers is given by 
			\begin{align*}
				\EuScript{L}^{(r)}(t)=\sum_{i=1}^{r}e^{it}{r \brace i} \left(\alpha^{i+1} e^{\alpha (e^t-1)}+\beta^{i+1} e^{\beta (e^t-1)}\right).
			\end{align*}
		\end{theorem}
		Since the proof is similar to the proof of Theorem \ref{th:deivate_Fibo},  we leave it to the reader.

		\section{Fibonacci-Fubini triangle}

		In this section, we examine the so-called Fibonacci-Fubini triangle $\T$ with terms $\Tt{n}{k}=F_{k+1}\stirling{n}{k}$  for $1\leq n$ and $1\leq k \leq n$, when we arrange the terms into a Pascal's-like arithmetical triangle. Table~\ref{tab:Triangle_FS} shows its first few rows with rows index $n$. (In an extended version, let the terms $\Tt{n}{0}$ for all $n\ge1$ be zeros, and $\Tt{0}{0}$=1.)
		Obviously	
		\[  \mathcal{F}_n=\sum_{k=0}^n \Tt{n}{k},
		\]
		where $(\mathcal{F}_n)$ is the Fibonacci-Fubini sequence and the sum of rows sequence of the triangle $\T$ as well.
		
		\begin{table}[!ht]
			\centering
			\scalebox{0.99}{ \begin{tikzpicture}[->,xscale=1.2,yscale=0.6, auto,swap]
					\node(a00) at (0,0)    {1};
					
					\node (a01) at (-0.5,-1)   {1};
					\node (a02) at (0.5,-1)   {2};
					
					\node (a01) at (-1,-2)   {1};
					\node (a02) at (0,-2)   {6};	
					\node (a01) at (1,-2)   {3};
					
					\node (a02) at (-1.5,-3)   {1};		
					\node (a01) at (-0.5,-3)   {14};
					\node (a01) at (0.5,-3)   {18};
					\node (a02) at (1.5,-3)   {5};	
					
					\node (a01) at (-2,-4)   {1};
					\node (a01) at (-1,-4)   {30};
					\node (a01) at (0,-4)   {75};
					\node (a02) at (1,-4)   {50};
					\node (a02) at (2,-4)   {8};
					
					\node (a01) at (-2.5,-5)   {1};		
					\node (a01) at (-1.5,-5)   {62};
					\node (a02) at (-0.5,-5)   {270};
					\node (a02) at (0.5,-5)   {325};
					\node (a01) at (1.5,-5)   {120};
					\node (a01) at (2.5,-5)   {13};
					
					\node (a01) at (-3,-6)   {1};
					\node (a02) at (-2,-6)   {126};
					\node (a02) at (-1,-6)   {903};
					\node (a01) at (0,-6)   {1750};
					\node (a01) at (1,-6)   {1120};
					\node (a02) at (2,-6)   {273};
					\node (a01) at (3,-6)   {21};
					
					\node (a01) at (-3.5,-7)   {1};
					\node (a01) at (-2.5,-7)   {254};		
					\node (a01) at (-1.5,-7)   {2898};
					\node (a02) at (-0.5,-7)   {8505};
					\node (a02) at (0.5,-7)   {8400};
					\node (a01) at (1.5,-7)   {3458};
					\node (a01) at (2.5,-7)   {588};
					\node (a01) at (3.5,-7)   {34};
					
					\node (a01) at (-4,-8)   {1};
					\node (a01) at (-3,-8)   {510};
					\node (a02) at (-2,-8)   {9075};
					\node (a02) at (-1,-8)   {38850};
					\node (a01) at (0,-8)   {55608};
					\node (a01) at (1,-8)   {34398};
					\node (a02) at (2,-8)   {9702};
					\node (a01) at (3,-8)   {1224};
					\node (a01) at (4,-8)   {55};	
			\end{tikzpicture}}
			\caption{Fibonacci-Fubini triangle $\T$.} 
			\label{tab:Triangle_FS}
		\end{table}
		
		Now we give a recurrence formula which shows the connection between a term and its neighbors in the row before. 
		\begin{theorem} For $2\leq n$ and $1\leq k \leq n$, we have the recurrence formula
			\begin{equation}\label{eq:FS_recu}
				\Tt{n}{k}= \frac{F_{k+1}}{F_{k}}\, \Tt{n-1}{k-1} + k\, \Tt{n-1}{k}.
			\end{equation} 
		\end{theorem}
		\begin{proof}
			From the recurrence of the Stirling numbers of the second kind, we have 
			\[ F_{k}F_{k+1}	{n \brace k} = F_{k}F_{k+1}	{n-1 \brace k-1} + F_{k}F_{k+1}	k\, {n-1 \brace k}, \]
			and it immediately implies  the result.
		\end{proof}

		In this section,  we examine some properties of the triangle $\Tt{n}{k}$.

		\subsection{Left-down diagonal sequences of the triangle \texorpdfstring{$\T$}{TFF} }
		
		Now we examine the left-down diagonal sequences of Fibonacci-Fubini triangle given in Table~\ref{tab:Triangle_FS} when $k$ is fixed and $n$ is considered to be a variable. Let we denote these sequences by $(\Tt{n}{k})_{n=k}^\infty=    (\Tt{n}{k})_{n=k}^\infty$. Then the first few sequences are 
		\begin{equation*}
			\begin{aligned}
				(\Tt{n}{1}) &= (1,1,1,1,1,1,1,1,1,1,1,1,1, \ldots) &=&\ \text{A000012}, \\ 
				(\Tt{n}{2}) &= (2, 6, 14, 30, 62, 126, 254, 510, 1022, 2046,  \ldots) &=&\ \text{A000918}, \\ 
				(\Tt{n}{3}) &= (3, 18, 75, 270, 903, 2898, 9075, 27990, 85503,  \ldots)&=&\ \text{A094033},\\
				(\Tt{n}{4}) &= (5, 50, 325, 1750, 8505, 38850, 170525, 728750,  \ldots)&&\ \text{not in the OEIS},\\
				(\Tt{n}{5}) &= (8, 120, 1120, 8400, 55608, 340200, 1973840,   \ldots)&&\ \text{not in the  OEIS}.
			\end{aligned}
		\end{equation*} 
		
		\begin{theorem}\label{th:down1}
			The left-down diagonal sequences $(\Tt{n}{k})_{n=k}^\infty$ in Fibonacci-Fubini triangle are recursively given by the $k$th order constant coefficients homogeneous linear recurrence relations
			\begin{equation}\label{eq:recu_t}
				0 = \sum_{j=0}^{k} \stirlingI{k+1}{k-j+1}\Tt{n-j}{k}, \quad n\geq k+1,
			\end{equation}
			where the coefficients are the Stirling numbers of the first kind (see the triangle A008276  in the OEIS \cite{oeis}).  
		\end{theorem}
		\begin{proof} Recall the identities for the Stirling numbers of the first kind for $n\geq 1$  that  ${n \brack k} = {n-1 \brack k-1} - (n-1){n-1 \brack k}$ for $k\geq1$, moreover, ${n \brack 0}=0$, ${n \brack 1}=(-1)^{n-1}(n-1)!$, ${n \brack n}=1$, and  ${n \brack k}=0$, if $k>n$.
			
			We prove the statement by induction on $k$.
			If $k=1$, then  $\sum_{j=0}^{1} {2 \brack 2-j}\Tt{n-j}{k}= \Tt{n}{1}-\Tt{n-1}{1}=0$, where $n\geq 2$.
			
			Suppose that \eqref{eq:recu_t} holds in the case $k-1$, so 
			\[ 0 = \sum_{j=0}^{k-1} {k \brack k-j}\Tt{n-j}{k-1}=\sum_{j=0}^{k-1} {k \brack k-j}F_k {n-j \brace k-1} , \text{ for } n\geq k.\]
			Indeed
			\begin{equation} \label{eq:hyp01}
				\sum_{j=0}^{k-1} {k \brack k-j} {n-j \brace k-1} =0
			\end{equation} also holds.  
			Using  the induction hypothesis  \eqref{eq:hyp01} we obtain
			\begin{eqnarray*}
				\sum_{j=0}^{k} {k+1 \brack k-j+1}\Tt{n-j}{k} &=& 
				\sum_{j=0}^{k} {k+1 \brack k-j+1}\Tt{n-j}{k}= \sum_{j=0}^{k} {k+1 \brack k-j+1}F_{k+1} {n-j \brace k} \\
				&=&  F_{k+1}   \sum_{j=0}^{k} \left( {k \brack k-j}- k {k \brack k-j+1} \right) {n-j \brace k} \\
				&=&  F_{k+1}  \left( \sum_{j=0}^{k-1} {k \brack k-j}{n-j \brace k} - k  \sum_{j=1}^{k} {k \brack k-j+1}  {n-j \brace k}  \right)\\
				&=&  F_{k+1}  \left( \sum_{j=0}^{k-1} {k \brack k-j}{n-j \brace k} - k  \sum_{j=0}^{k-1} {k \brack k-j}  {n-j-1 \brace k}  \right)\\
				&=&  F_{k+1}  \left( \sum_{j=0}^{k-1} {k \brack k-j} \left({n-j \brace k} - k   {n-j-1 \brace k}\right)  \right)\\
				&=&  F_{k+1} \sum_{j=0}^{k-1} {k \brack k-j}{n-j \brace k-1}  =0.
			\end{eqnarray*}  
		\end{proof}
		
		The recurrence relations of the first few sequences are 
		\begin{equation*}
			\begin{aligned}
				\Tt{n}{1} &= \Tt{n-1}{1} , \\ 
				\Tt{n}{2} &= 3\Tt{n-1}{2}-2\Tt{n-2}{2}, \\ 
				\Tt{n}{3} &= 6\Tt{n-1}{3}-11\Tt{n-2}{3}+6\Tt{n-3}{3},\\
				\Tt{n}{4} &= 10\Tt{n-1}{4}-35\Tt{n-2}{4}+50\Tt{n-3}{4}-24\Tt{n-4}{4},\\
				\Tt{n}{5} &= 15\Tt{n-1}{5}-85\Tt{n-2}{5}+225\Tt{n-3}{5}-27\Tt{n-4}{5} +120\Tt{n-5}{5}.
			\end{aligned}
		\end{equation*} 
		

			Since the Stirling numbers of the first kind are the coefficients in the expansion of the falling factorial 
			$(x)_n = x(x-1)(x-2)\cdots (x-n+1)$, then Theorem~\ref{th:down1} implies the following corollary.
			\begin{corollary}
				The characteristic polynomial $p^{(k)}(x)$ of recurrence sequence $(\Tt{n}{k})_{n=k}^\infty$ is 
				\[p^{(k)}(x)=(x-1)(x-2)\cdots (x-k).\]
			\end{corollary}
			
			Therefore, our proof of Theorem~\ref{th:down1} is independent of the Fibonacci numbers, then  we obtain a similar proof a new and interesting connection between the Stirling numbers of the first and the second kind, as follows. 
			\begin{corollary}
				Let $\left( {n \brace k} \right)_{n=k}^\infty$ be the $k$th left-down diagonal sequences in triangle of Stirling numbers of the second kind, A008277 in the OEIS, then it is a $k$th order constant coefficients homogeneous linear recurrence relation
				\begin{equation*}
					0 = \sum_{j=0}^{k} {k+1 \brack k-j+1}{n-j \brace k}, \quad n\geq k+1,
				\end{equation*}
				where the coefficients are the Stirling numbers of the first kind (see the triangle A008276  in the OEIS).  
			\end{corollary}

			\begin{corollary}
				Using identity \eqref{eq:FS_recu}  for  sequence $\Tt{n}{k}$   we obtain
				\begin{equation*}
					F_{k}\,	\Tt{n}{k} = F_{k+1}\, \Tt{n-1}{k-1}+ k F_{k}\, \Tt{n-1}{k}, \quad \text{ for } k\ge2,\, n\geq k+1.
				\end{equation*}
			\end{corollary}

			\subsection{Right-down diagonal sequences of the triangle \texorpdfstring{$\T$}{TFF} }

			Now we examine the right-down diagonal sequences 
			$(d_n^{(r)})_{n=r}^\infty= (\Tt{n}{n-r+1})_{n=r}^\infty$ 
			for all $r\geq 1$. The first few sequences are 
			\begin{equation*}
				\begin{aligned}
					(d_n^{(1)}) &= (1, 2, 3, 5, 8, 13, 21, 34, 55, 89, 144, 233, 377 \ldots) &=&\ \text{A000045}, \\ 
					(d_n^{(2)}) &= (1, 6, 18, 50, 120, 273, 588, 1224, 2475, 4895,   \ldots) &=&\ \text{A086926}, \\ 
					(d_n^{(3)}) &= (1, 14, 75, 325, 1120, 3458, 9702, 25500, 63525,   \ldots)&=&\ \text{not in the  OEIS},\\
					(d_n^{(4)}) &= (1, 30, 270, 1750, 8400, 34398, 123480, 403920,  \ldots)&=&\ \text{not in the OEIS},\\
					(d_n^{(5)}) &= (1,62,903,8505,55608,296751,1343727,5406918,  \ldots)&=&\ \text{not in the OEIS},\\
					(d_n^{(6)}) &= (1, 126, 2898, 38850, 340200, 2333331, 13175316,    \ldots)&=&\ \text{not in the OEIS}.
				\end{aligned}
			\end{equation*} 
			
			Using identity \eqref{eq:FS_recu}  for  sequence $d_n^{(r)}$ with $k=n-r+1$  we obtain
			\begin{equation}\label{cica}
				d_n^{(r)} = \frac{F_{n-r+2}}{F_{n-r+1}}\,d_{n-1}^{(r)} + (n-r+1) d_{n-1}^{(r-1)}, \quad \text{ for } r\ge2,\, n\geq r+1.
			\end{equation}
			
			Their recurrence relations (as we see later) are 
			\begin{equation*}
				\begin{aligned}
					d_n^{(1)} &= d_{n-1}^{(1)}+d_{n-2}^{(1)}, \\ 
					d_n^{(2)} &= 3d_{n-1}^{(2)}-5d_{n-3}^{(2)}+3d_{n-5}^{(2)}+d_{n-6}^{(2)}, \\ 
					d_n^{(3)} &= 5d_{n-1}^{(3)}-5d_{n-2}^{(3)} -10d_{n-2}^{(3)}+15d_{n-4}^{(3)} +11d_{n-5}^{(3)} -15d_{n-6}^{(3)} -10d_{n-7}^{(3)} \\
					& \hspace{1.5cm}  +5d_{n-8}^{(3)} +5 d_{n-9}^{(3)}+d_{n-10}^{(3)},
				\end{aligned}
			\end{equation*}
			where the coefficients are in Table~\ref{fig:Triangle_coeff2}, or the terms of every second odd rows of the reverse of triangle A084610 multiplied by $-1$. (The triangle A084610 is read by rows, where the $n$-th row lists the $2n+1$ coefficients of $(1+x-x^2)^n$.)
			\begin{theorem}
				The characteristic polynomial of recurrence sequence $(d_n^{(r)})$ is \[q^{(r)}(x)=(x^2-x-1)^{2r-1}.\]
			\end{theorem}
			\begin{proof}
				The zeros $\alpha$ and $\beta$ of $q^{(r)}(x)$ both have multiplicity $2r-1$.
				Note that $\beta=(1-\sqrt{5})/2$ is the algebraic conjugate of $\alpha=(1+\sqrt{5})/2$. Hence it is sufficient to prove that $d_n^{(r)}$ can be written in the form
				\[
				d_n^{(r)}=p_r(n)\alpha^{n+2-r}+\bar{p}_r(n)\beta^{n+2-r},
				\]
				where $p_r(n)$ is a polynomial of $n$ with coefficients from $\mathbb{Q}[\alpha]$, $\deg(p_r)=2r-2$,	and $\bar{p}_r(n)$ is conjugate of $p_r(n)$ (i.e.,~$\alpha$ is replaced by $\beta$ in the coefficients).
				
				We will use the technique of induction. Obviously, the statement is true for $r=1$. Indeed, 
				\[d_n^{(1)}= \Tt{n}{n}=F_{n+1}{n \brace n}=F_{n+1}=\frac{1}{\sqrt{5}}\alpha^{n+1}-\frac{1}{\sqrt{5}}\beta^{n+1}\quad(n\ge1),\]		
				that is $p_1(n)=1/\sqrt{5}$ is a constant polynomial ($\deg(p_1)=0$), $\bar{p}_1(n)=-1/\sqrt{5}$.
				
				Assume now that the theorem is true for some $r\ge1$. We will prove that there exist polynomial $p_{r+1}(n)$ with the given conditions above. In order to do that we are looking for $d_{r+1}(n+1)$ in the form
				\[
				d_{n+1}^{(r+1)}=p_{r+1}(n+1)\alpha^{n+2-r}+\bar{p}_{r+1}(n+1)\beta^{n+2-r},
				\]
				and will show the (unique) existence of it.
				Multiply (\ref{cica}) by $F_{n+1-r}$, and use the induction hypothesis to obtain
				\begin{multline*}
					F_{n+1-r}\left(p_{r+1}(n+1)\alpha^{n+2-r}+\bar{p}_{r+1}(n+1)\beta^{n+2-r}\right)=\\  F_{n+2-r}\left(p_{r+1}(n)\alpha^{n+1-r}+\bar{p}_{r+1}(n)\beta^{n+1-r}\right)\\
					+(n+1-r)F_{n+1-r}\left(p_{r}(n)\alpha^{n+2-r}+\bar{p}_{r}(n)\beta^{n+2-r}\right).
				\end{multline*}
				Recall that $\alpha\beta=-1$. Eliminate the coefficients of the terms $\alpha^{2n+3-2r}$, $\beta^{2n+3-2r}$, and $(-1)^{n+1-r}$, respectively. We obtain
				\begin{equation*}
					\begin{aligned}
						p_{r+1}(n+1)&=\phantom{-}  p_{r+1}(n)+(n+1-r)p_r(n),\\
						-\bar{p}_{r+1}(n+1)&=-\bar{p}_{r+1}(n)-(n+1-r)\bar{p}_r(n),\\
						\bar{p}_{r+1}(n+1)\beta-{p}_{r+1}(n+1)\alpha&=-{p}_{r+1}(n)\beta+\bar{p}_{r+1}(n)\alpha+(n+1-r)\bar{p}_r(n)\beta\\
						&\qquad\qquad -(n+1-r){p}_r(n)\alpha,
					\end{aligned}
				\end{equation*}
				respectively. The first two equalities are equivalent since they are the conjugates of each other. The first information leads to
				$(n+1-r)p_r(n)=p_{r+1}(n+1)-p_{r+1}(n)$. Plug it into the third equality. After straightforward calculations we have
				\[
				(\alpha-\beta)(p_{r+1}(n)+\bar{p}_{r+1}(n))=0,
				\]
				subsequently $\bar{p}_{r+1}(n)=-p_{r+1}(n)$ follows since $\alpha-\beta=\sqrt{5}\ne0$.
				
				Thus we need to consider only the first equality 	$p_{r+1}(n+1)=p_{r+1}(n)+(n+1-r)p_r(n)$ again. Clearly, it is equivalent to
				\begin{equation}\label{marc4}
					p_{r+1}(n+1)-p_{r+1}(n)=(n+1-r)p_r(n),
				\end{equation}
				and from the theory of difference equations we know that this condition, together with the "final condition" $p_{r+1}(r+1)=1/\sqrt{5}$ it provides uniquely the polynomial $p_{r+1}(n)$. Since $\deg(p_r)=2r-2$ and $(n+1-r)$ is a linear factor, then the degree of $p_{r+1}$ is $2r$.
			\end{proof} 
			
			\begin{example}
				The first few polynomials are
				\begin{eqnarray*}
					p_1(n)&=&\frac{1}{\sqrt{5}},\\
					p_2(n)&=&\frac{1}{2\sqrt{5}}(n-1)n,\\
					p_3(n)&=&\frac{1}{24\sqrt{5}}(3n-5)(n-2)(n-1)n.
				\end{eqnarray*}
				In this example, knowing $p_2(n)$ we illustrate how $p_3(n)$ can be obtained.
				Since $\deg(p_3(n))=2\cdot3-2=4$, we will find $p_3(n)$ in the form $p_3(n)=an^4+bn^3+cn^2+dn+e$. Thus
				\begin{equation}\label{r1}
					p_3(n+1)-p_3(n)=4an^3+(6a+3b)n^2+(4a+3b+2c)n+(a+b+c+d).
				\end{equation}
				On the other hand,
				\begin{equation}\label{r2}
					(n+1-r)p_r(n)=(n-1)p_2(n)=(n-1)\cdot\frac{1}{2\sqrt{5}}(n-1)n=
					\frac{1}{2\sqrt{5}}n^3-\frac{1}{\sqrt{5}}n^2+\frac{1}{2\sqrt{5}}n.
				\end{equation}
				Comparing the right hand-sides of (\ref{r1}) and (\ref{r2}), and solving the corresponding system of linear equations we conclude
				\[
				a=\frac{1}{8\sqrt{5}},\quad b=-\frac{7}{12\sqrt{5}},\quad c=\frac{7}{8\sqrt{5}},\quad d=-\frac{5}{12\sqrt{5}}.
				\]
				Subsequently,
				\[
				p_3(n)=\frac{1}{24\sqrt{5}}\left(3n^4-14n^3+21n^2-10n+e\right),
				\]
				and the "final condition" admits $e=0$.
				Using this fact, finally we get the decomposition of $p_3(n)$ as it is given in the list above.
			\end{example}

			\begin{table}[!ht]
				\centering
				\scalebox{0.99}{ \begin{tikzpicture}[->,xscale=0.85,yscale=0.65, auto,swap]
						\node(a00) at (0,0)    {1};
						
						\node (a01) at (-1,-1)   {1};
						\node (a02) at (0,-1)   {$-1$};	
						\node (a01) at (1,-1)   {$-1$};
						
						\node (a01) at (-3,-2)   {1};
						\node (a02) at (-2,-2)   {$-3$};
						\node (a02) at (-1,-2)   {0};
						\node (a01) at (0,-2)   {5};
						\node (a01) at (1,-2)   {0};
						\node (a02) at (2,-2)   {$-3$};
						\node (a01) at (3,-2)   {$-1$};
						
						\node (a01) at (-5,-3)   {1};
						\node (a01) at (-4,-3)   {$-5$};
						\node (a01) at (-3,-3)   {5};
						\node (a02) at (-2,-3)   {10};
						\node (a02) at (-1,-3)   {$-15$};
						\node (a01) at (0,-3)   {$-11$};
						\node (a01) at (1,-3)   {15};
						\node (a02) at (2,-3)   {10};
						\node (a01) at (3,-3)   {$-5$};
						\node (a01) at (4,-3)   {$-5$};	
						\node (a01) at (5,-3)   {$-1$};	
						
						\node (a01) at (-7,-4)   {1};
						\node (a01) at (-6,-4)   {$-7$};
						\node (a01) at (-5,-4)   {14};
						\node (a01) at (-4,-4)   {$7$};
						\node (a01) at (-3,-4)   {$-49$};
						\node (a02) at (-2,-4)   {14};
						\node (a02) at (-1,-4)   {$77$};
						\node (a01) at (0,-4)   {$-29$};
						\node (a01) at (1,-4)   {$-77$};
						\node (a02) at (2,-4)   {14};
						\node (a01) at (3,-4)   {$49$};
						\node (a01) at (4,-4)   {$7$};	
						\node (a01) at (5,-4)   {$-14$};	
						\node (a01) at (6,-4)   {$-7$};	
						\node (a01) at (7,-4)   {$-1$};	            
				\end{tikzpicture}}
				\caption{Triangle of coefficients} 
				\label{fig:Triangle_coeff2}
			\end{table}

			\section{Combinatorial interpretation of Fibonacci-Fubini numbers}
			
			We give a combinatorial problem of packing goods on a train in two different ways, the solution of which is given by the Fibonacci-Fubini numbers. Suppose that we have $n$ items of goods, which we distribute into $k$ boxes and load the boxes onto a train of containers as follows.
			
			Given $n$ and $k$ identical and labeled objects and not labeled boxes, respectively. Put the objects into the boxes so that each box contains at least one object. Thus, $k\leq n$. Then sort the boxes in order of the smallest number of elements and label the boxes according to that.
			
			If the empty boxes were originally labeled and ordered, then the objects are placed in the boxes by putting the first one in the first box; the second one either in the first or in the second box. We always put the next object in the sequence into one of the non-empty boxes, or into the next labeled non-empty box. This is also the way to get the sequence of previous filled boxes.  Figure~\ref{fig:Fibo_Stir_01} shows the second type of boxing.

			We also have a $k$-box-long train and on it some $2$-box-long containers. Load the boxes on the train in order from the front of the train, but if there is a container next in line, put the box in the back of the container (e.g., the box always slides backwards in the container).  
			
			It is easy to see that the number of possible trains with containers and boxes having objects is $\Tt{n}{k}$, the sum considering all $k$ from $1$ to $n$ is the Fibonacci-Fubini number $\mathcal{F}_n$.  In Figure~\ref{fig:Fibo_Stir_train01}, there is a container in the front of the train, in Figure~\ref{fig:Fibo_Stir_train02} there is not. 
			
			\begin{figure}[!ht]
				\centering
				\includegraphics[scale=0.85]{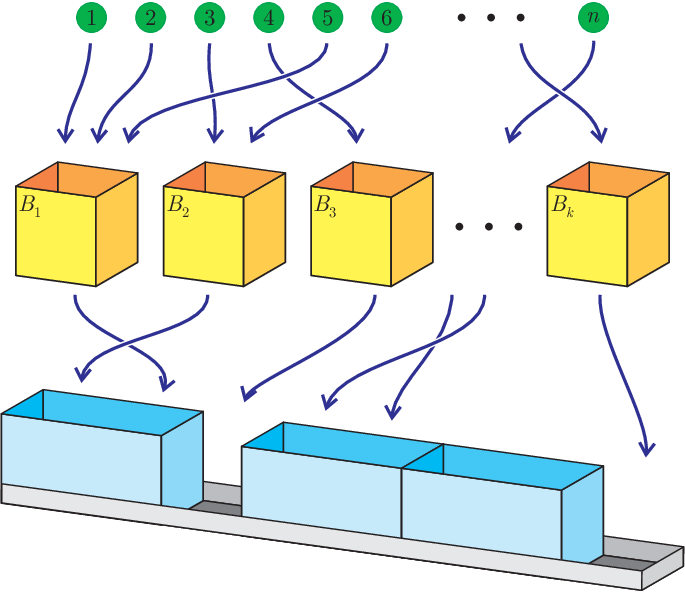} 
				\caption{Combinatorial interpretation of Fibonacci-Fubini numbers}
				\label{fig:Fibo_Stir_01}
			\end{figure}
			
			\begin{figure}[!ht]
				\centering
				\includegraphics[scale=0.85]{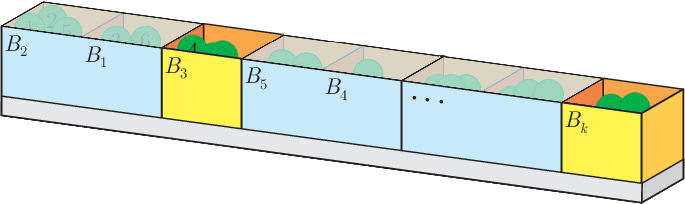} 
				\caption{Combinatorial interpretation: boxes on the train}
				\label{fig:Fibo_Stir_train01}
			\end{figure}
			
			\begin{figure}[!ht]
				\centering
				\includegraphics[scale=0.85]{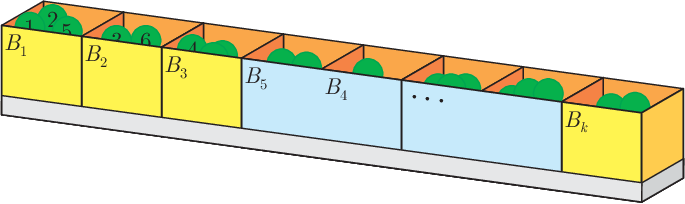} 
				\caption{Combinatorial interpretation: 2nd version of the boxes on the train}
				\label{fig:Fibo_Stir_train02}
			\end{figure}

			\begin{example}We show the connection between the permutations of ordered partitions of the set $[n]=\{1,2,\ldots, n\}$ for $1\leq n\leq 4$ and their combinatorial interpretation. In cases of $n=2$ and $n=3$, the first rows of the tables show the partitions, and the second ones do their appropriate tillings (the trains with boxes and containers).  When $n=4$ we present the connection in an other structure.  
				
				\begin{itemize}
					\item 	For $n=1$,  $\Tt{1}{1}=1$ and  $\mathcal{F}_1=1$. 
					\item 	For $n=2$,  $\Tt{2}{1}=1$, $\Tt{2}{2}=2$, and  $\mathcal{F}_2=3$. 
					\begin{equation*}
						\begin{aligned}
							{12};&&{1|2};&&{2|1}, \\ 
							\squ;&&\squ\, \squ; & &\rec .
						\end{aligned}
					\end{equation*}
					\item For $n=3$,  $\Tt{3}{1}=1$, $\Tt{3}{2}=6$, $\Tt{3}{3}=3$, and  $\mathcal{F}_3=10$.
					\begin{equation*}
						\begin{aligned}
							{123};&&{1|23};&&{23|1};&&{12|3};&&{3|12};&&{13|2};&&{2|13};&&{1|2|3};&&{1|3|2};&&{2|1|3}, \\ 
							\squ;&&\squ\, \squ; & &\rec;&&\squ\, \squ; &&\rec;&&\squ\, \squ; &&\rec;&& \squ\, \squ\, \squ; &&\squ\, \rec;&&\rec\, \squ .
						\end{aligned}
					\end{equation*}
					\item 	For $n=4$,  $\Tt{4}{1}=1$, $\Tt{4}{2}=14$, $\Tt{4}{3}=18$, $\Tt{4}{4}=5$, and $\mathcal{F}_4=38$. 
					\begin{equation*}
						\begin{aligned}
							k=1 &:& \squ &:& 1234,\\
							k=2 &:& \squ\, \squ &:& 1|234;&& 12|34;&& 13|24;&& 14|23;&&  123|4;&&  124|3;&& 134|2,\\
							&& \rec &:& 234|1;&& 34|12;&& 24|13 ;&& 23|14 ;&&   4|123;&&  3|124;&& 2|134,\\
							k=3 &:& \squ\,\squ\,\squ  &:& 1|2|34;&& 1|24|3;&&14|2|3;&& 1|23|4;&&  13|2|4;&& 12|3|4, \\
							&& \squ\,\rec &:& 1|34|2;&& 1|3|24;&&14|3|2;&& 1|4|23;&&  13|4|2;&& 12|4|3, \\
							&& \rec\,\squ &:& 2|1|34;&& 24|1|3;&&2|14|3;&& 23|1|4;&&  2|13|4;&& 3|12|4, \\
							k=4 &:& \squ\,\squ\,\squ\,\squ   &:& 1|2|3|4, \\
							&& \squ\,\squ\,\rec &:& 1|2|4|3,\\
							&& \squ\,\rec\,\squ &:& 1|3|2|4,\\
							&& \rec\,\squ\,\squ &:& 2|1|3|4,\\
							&& \rec\,\rec &:& 2|1|4|3.\\
						\end{aligned}
					\end{equation*}
				\end{itemize}
			\end{example}


			\section*{Declarations}

			\noindent\textbf{Author Contributions.} The authors contributed equally to this manuscript.
			\medskip
			
			\noindent\textbf{Conflict of interest.} The authors declare that they have no conflict of interest.
			\medskip
			
			\noindent\textbf{Data availability.} No data was used for the research described in the article.
			\medskip
			
			\noindent\textbf{Funding.} 
			Y.D.: was partially supported by the DGRSDT grant N$^{\text{o}}$.\ C0656701. L.Sz.  was  supported by Hungarian National Foundation for Scientific Research Grant N$^{\text{o}}$.\ 130909.
			\medskip
			
			\noindent\textbf{Financial interests.} Authors declare they have no financial interests.
			\medskip
			
			\noindent\textbf{Declaration of Generative AI and AI-assisted technologies in the writing process.} During the preparation of this work the authors did not use Generative AI and AI-assisted technologies.
			
		
		%
		\bibliography{Fibo_Stirl_bib} 
		\bibliographystyle{plain}	

\end{document}